\documentclass[12pt]{article}

\usepackage{amsfonts,amsmath, amssymb}

\usepackage{xcolor}

\usepackage{enumerate}
\usepackage{epsf,epsfig,amsfonts,a4wide}
\usepackage{amsfonts, mathdots}
\usepackage{amsmath,amssymb,amsthm}
\usepackage{url}
\usepackage{amsmath,amssymb,amsthm}

\usepackage{amssymb}
\usepackage{amsthm}
\usepackage{epsf,epsfig,amsfonts,graphicx}

\newtheorem{Theorem}{Theorem}[section]

\newtheorem{Lemma}{Lemma}[section]

\newtheorem{Corollary}{Corollary}[section]

\newtheorem{Ex}{Example}[section]
\newtheorem{Remark}{Remark}[section]

\theoremstyle{remark}

\newcommand{\be}{\begin{equation}}
\newcommand{\ee}{\end{equation}}

\newcommand{\R}{\mathbb{R}}\newcommand{\Id}{\textrm{\rm Id}}

\newcommand{\gl}{\mathrm{gl}}

\newcommand{\ddd}{\mathrm{d}}

\newcommand{\pd}[2]{\frac{\partial#1}{\partial#2}}

\newcommand{\weg}[1]{}

\title{Nijenhuis geometry IV: conservation laws, symmetries and integration of certain non-diagonalisable systems of hydrodynamic type in quadratures}

\author{Alexey V. Bolsinov\footnote{ School of Mathematics,
 Loughborough University,
 LE11 3TU, UK \ \ 
 \quad {\tt A.Bolsinov@lboro.ac.uk} } \quad
\& \quad  Andrey Yu. Konyaev\footnote{Faculty of Mechanics and Mathematics, Moscow State University,  and  Moscow Center for Fundamental and Applied Mathematics,119992, Moscow Russia
 \ \ \quad {\tt  maodzund@yandex.ru}} \quad \& \quad Vladimir S. Matveev\footnote{
Institut f\"ur Mathematik, Friedrich Schiller Universit\"at Jena,
07737 Jena Germany  and  La Trobe University, Melbourne, Australia \ \ \quad {\tt  vladimir.matveev@uni-jena.de}} 
}

\date{April 2023}

\begin{document}

\maketitle

\begin{abstract}
The paper contains two lines of results: the first one is a study of symmetries and conservation laws of $\gl$-regular Nijenhuis operators. We prove the splitting Theorem for symmetries and conservation laws of Nijenhuis operators,   show that the space of symmetries of a $\gl$-regular Nijenhuis operator forms a commutative algebra with respect to  (pointwise) matrix multiplication. Moreover,    all  the  elements of this algebra are strong symmetries of each other. We establish a natural relationship between 
symmetries and conservation laws of a $\gl$-regular Nijenhuis operator and systems of the first and second companion coordinates. 
Moreover, we show that the space of conservation laws is naturally related to the space of symmetries in the sense that any conservation laws can be obtained from a single conservation law by multiplication  with an appropriate symmetry.  In particular, 
we provide an explicit description of all symmetries and conservation laws for $\gl$-regular operators at algebraically generic points. 

The second line of results contains an application of the theoretical part to a certain system of partial differential equations of hydrodynamic type, which was previously studied by different authors, but mainly in the diagonalisable case.
 We show that this system is integrable in quadratures, i.e., its solutions can be found for almost all initial curves  by integrating closed 1-forms and solving some systems of functional equations. The system is not diagonalisable in general,  and construction and integration of such systems is an actively studied  and  explicitly stated problem in the literature.
\end{abstract}


\tableofcontents

\section{Introduction}

\subsection{Foreword}

This work continues the research programme started in  \cite{nij,openprob} and aimed at a systematic study of Nijenhuis geometry and its applications. As in our previous publications in this area, we first develop the general theory of Nijenhuis manifolds and study those properties of Nijenhuis structures that, in our opinion, are the most natural and important \cite{nij, nij3, nij2}, and then apply the results thus obtained to solve interesting non-trivial problems in other areas of mathematics where Nijenhuis operators naturally arise \cite{nijapp, nijapp2, nijapp3, nijapp4, ortsepar}.

In this paper,  we focus on the interplay between the conservation laws and symmetries related to a field $L=(L^i_j)$ of endomorphisms in the special case when $L$ is a $\gl$-regular Nijenhuis operator.   If $L$ is hyperbolic, i.e.,  its eigenvalue are real and pairwise distinct,   the description of symmetries and conservation laws is straightforward and can be regarded as a folklore result.  However, all the other cases including Jordan blocks and points where the eigenvalues collide are more complicated and, to the best of our knowledge, have not been treated in detail.  
The only known results concern the existence of conservation laws for Nijenhuis operators  \cite{ cgm, gm, osborn1, osborn2,  turiel}.

The main results of the present paper are devoted to the spaces of symmetries/conservation laws rather than individual elements of these spaces.  Unless otherwise clearly said, the statements in our paper make sense both in the smooth category and in the real-analytic category.  If neither of these conditions is mentioned,  then both cases are equally allowed.  Some proofs, however,  require real analyticity since they essentially use the Cauchy-Kovalevskaya theorem for PDE systems. In such a situation, we will give the necessary comments. Our main results are as follows:  

\begin{itemize}

\item  We show that the space of symmetries of a $\gl$-regular Nijenhuis operator $L$ is a commutative algebra w.r.t. (pointwise) matrix multiplication,   all of its elements are strong symmetries for each other.   Locally in a neighbourhood  of a point $\mathsf p$ (regardless how singular the point $\mathsf p$ is), each symmetry is defined by $n$ real analytic function of one variable.   

\item  We establish a natural relationship between symmetries/conservation laws of $L$ and the systems of first/second companion coordinates. 

\item  We show that the space of conservation laws is naturally related to the space of symmetries in the sense that every conservation law can be obtained from a single conservation law by the `usual matrix multiplication' with a suitable symmetry.  
  
\item  We provide an explicit description of all the symmetries and conservation laws for $\gl$-regular operators at algebraically generic points (including the case of Jordan blocks and complex eigenvalues). 

\item We show that the $\gl$-regularity condition is essential. We give a couple of explicit examples demonstrating that the `good' properties of symmetries and conservation laws listed above do not hold, if this condition is omitted.

\end{itemize}

Next, we apply these results to {\it integrate in quadratures}  some quasilinear systems of type  $u_t = A(u) u_x$ (notice that the field $A(u)$ of endomorphisms is not Nijenhuis but, as it often happens in integrable systems, is naturally related to a certain Nijenhuis operator $L$, see \cite{m}, \cite{ml}).  In the hyperbolic case, this construction is due to E.~Ferapontov \cite{fer, fer1}, but we generalise it to the case of Jordan blocks and complex eigenvalues, including singular points where the algebraic type of $A(u)$ changes. To the best of our knowledge,  this is the first construction that provides integration in quadratures for non-diagonal quasilinear systems.  

It is worth noticing that our construction is universal in the sense that it can be equally applied to any type of  $\gl$-regular Nijenhuis operators. Also its conceptual part is essentially invariant, i.e., is not linked to any specific coordinate system (of course, for specific computations,  one has to use a certain coordinate system and we explain how do it).

{\bf Acknowledgements.} Vladimir Matveev  thanks the DFG for the support (grant  MA 2565/7). Some of results were obtained during a long-term research visit of VM to La Trobe University supported by Sydney Mathematics Research Institute  and ARC Discovery Programme   DP210100951.  The authors are grateful to M.\,Blaszak, E.\,Ferapontov, K.\,Marciniak and A.\,Panasyuk for discussion and useful comments.

\subsection{Basic definitions}

A closed differential form $\theta$ is said to be a {\it conservation law} of an operator field $L$ if $\ddd (L^* \theta)  = 0$, i.e., if the form $L^*\theta$ is closed.   We are mainly interested in local description of conservation laws and for this reason we will often think of $\theta$ as the differential of a function $f$, that is, $\theta = \ddd f$.  

For a pair of commuting operator fields $L$ and $M$ (i.e., such that $LM = ML$) one can define a $(1, 2)$ tensor field $\langle L, M \rangle$ by
$$
\langle L, M\rangle(\xi, \eta) = M[L\xi, \eta] + L[\xi, M\eta] - [L\xi, M\eta] - LM[\xi, \eta].
$$

This definition is due to A.\,Nijenhuis \cite[formula 3.9]{nijenhuis}. Notice that $\langle L,M\rangle + \langle M,L\rangle$ coincides with the Fr\"olicher-Nijenhuis bracket of $L$ and $M$ and makes sense for all operator fields (not necessarily commuting).  

We say that $M$ is a {\it symmetry} of an operator field $L$, if $LM = ML$ and the symmetric (in lower indices) part of $\langle L, M \rangle$ vanishes, that is, $\langle L, M \rangle (\xi, \xi) = 0$ for all vector fields $\xi$. Furthermore, we say that $M$ is  a {\it strong symmetry} of $L$, if $ML = LM$ and the entire tensor $\langle M, L \rangle$ vanishes. 

An operator field $L$ is called {\it Nijenhuis}, if its Nijenhuis torsion
$$
\mathcal N_L (\xi, \eta) =  L^2[\xi, \eta] - L [L\xi, \eta] - L[\xi, L\eta] + [L\xi, L\eta] 
$$
vanishes for all vector fields $\xi, \eta$. One can see that $\mathcal N_L = -\langle L, L \rangle$ and this tensor is skew-symmetric. Hence, $L$ is always a symmetry of itself, and is a strong symmetry if and only if $L$ is Nijenhuis.

A linear operator $L$ on a vector space $V$ of dimension $n$ is said to be {\it $\gl$-regular} if any of the following equivalent conditions holds:
\begin{enumerate}
    \item The operators $\operatorname{Id}, L, \dots, L^{n - 1}$ form a basis of the centraliser of $L$, i.e., of the vector space of operators commuting with $L$.
    \item  Each eigenvalue of $L$ has geometric multiplicity one. 
    \item There exists a vector $\xi\in V$ such that $\xi, L\xi, \dots, L^{n - 1} \xi$ are linearly independent (such a vector is called {\it cyclic}).
    \item There exists a $1$-form $\alpha\in V^*$ such that $\alpha, L^*\alpha, \dots, L^{* (n - 1)} \alpha$ are linearly independent.
\end{enumerate}

An operator field $L$ is {\it $\gl$-regular} if it is $\gl$-regular at every point.  Basically,  this condition means that eigenvalues of $L$ may collide, but these collisions  must be generic in the sense that rank of $L-\lambda\,\Id$ drops by one at most even at collision points. 

In Nijenhuis geometry,  collision points are considered to be singular.  More precisely, we say that a point $\mathsf p$ is {\it algebraically generic}  (or, equivalently, an operator field $L$ is {\it algebraically generic} at $\mathsf p$) it the algebraic structure of $L$ (Segre characteristic) does not change at some neighborhood of $p$.  Otherwise, $\mathsf p$ is called {\it singular} (see more details in \cite{nij}).  In the case of $\gl$-regular operators, algebraic genericity of $\mathsf p$ means that the multiplicities of all eigenvalues are locally constant near this point.  
Conversely,  at singular points of $\gl$-regular operators the multiplicities of some eigenvalues change. 

\subsection{Main Results:  General Theory} 

\begin{Theorem}\label{t1}
Let $L$ be a  Nijenhuis operator. Assume that at a point $\mathsf p$, its characteristic polynomial  $\chi_L(\lambda)=\det(\lambda\,\Id - L(\mathsf p))$ is factorised as $\chi_L(\lambda) = \chi_1(\lambda)\chi_2(\lambda)$, where $\chi_1(\lambda)$ and $\chi_2(\lambda)$ are coprime monic polynomials. Then in a~neighbourhood of $\mathsf p$, there exists a coordinate system 
$$
\underbrace{u^1_1, \dots, u^{m_1}_1}_{u_1}, \underbrace{u^1_2, \dots, u^{m_2}_2}_{u_2},  \quad \text{where} \quad m_1 = \operatorname{deg} \chi_1(\lambda),  m_2 = \operatorname{deg} \chi_2(\lambda),
$$
such that
\begin{enumerate}
    \item The Nijenhuis operator $L$ has the form
    $$
    L(u_1, u_2) = \begin{pmatrix}
         L_1(u_1) & 0  \\
         0 & L_2(u_2) 
     \end{pmatrix},
    $$
    where each of $L_i$ is Nijenhuis,  $i = 1, 2$. Moreover, $\chi_{L_1}(\lambda) = \chi_1(\lambda)$ and $\chi_{L_2}(\lambda) = \chi_2(\lambda)$.
    \item Every conservation law $\ddd f$ of $L$ has the form $\ddd\bigl( f_1(u_1) + f_2(u_2)\bigr)$, where $\ddd f_i$ is a conservation law of $L_i$, $i = 1, 2$.
    \item Every symmetry $M$ of $L$ has form
    $$
    M(u_1, u_2) =  \begin{pmatrix}
         M_1(u_1) & 0  \\
         0 & M_2(u_2) 
    \end{pmatrix},
    $$
    where $M_i$ is a symmetry of $L_i$,  $i = 1,2$.
    \item Every strong symmetry $M$ of $L$ has the form
    $$
    M(u_1, u_2) =  \begin{pmatrix}
         M_1(u_1) & 0  \\
         0 & M_2(u_2) 
     \end{pmatrix},
    $$
    where $M_i$ is a strong symmetry of $L_i$, $i = 1,2$.
\end{enumerate}
\end{Theorem}
The first statement of Theorem \ref{t1} is the splitting theorem for Nijenhuis operators, which was proved in this form in \cite[Theorem 3.1]{nij}, see also \cite[Theorems 1 and 2]{splitting}. Notice that within each block, the algebraic structure may vary from point to point as $\mathsf p$ is not supposed to be algebraically generic.  If $L$ is algebraically generic at $\mathsf p$, then the first statement of Theorem \ref{t1} was known much earlier, see e.g. \cite{gm},\cite{turiel}.

We say that $L$ is in {\it first companion form}, if its matrix in local coordinates is
\begin{equation}
\label{eq:comp1}
L_{\mathsf{comp1}} = \left(\begin{array}{ccccc}
     \sigma_1 & 1 & 0 & \dots & 0  \\
     \sigma_2 & 0 & 1 & \dots & 0  \\
     &  &  & \ddots &   \\
     \sigma_{n - 1} & 0 & 0 & \dots & 1 \\
     \sigma_n & 0 & 0 & \dots & 0  \\
\end{array}\right),
\end{equation}
The corresponding coordinates are called {\it first companion coordinates}. The functions $\sigma_i$ are the coefficients of the characteristic polynomial  
\begin{equation}
\label{eq:charpol}
\chi_L(\lambda)=\det(\lambda\,\Id - L)=\lambda^n - \sigma_1\lambda^{n-1} - \dots - \sigma_n,
\end{equation}  
and  $L_{\mathsf{comp1}}$ is Nijenhuis if and only if they satisfy a certain system of linear PDEs \cite[Theorem 1.1]{nij3}.

It was shown in \cite[Theorem 1.1]{nij3} that in the real analytic category,   first companion coordinates exist for every $\gl$-regular Niejnhuis operator near every point $\mathsf p$.  In the smooth category, one can construct such coordinates for algebraically generic points $\mathsf p$ by using the forth statement of Theorem \ref{t2_1} in combination with Theorems \ref{t1} and \ref{jordan}. However, near singular points  the existence of the first companion form is known only for certain examples and in general remains an interesting open problem. 

If $L$ is $\gl$-regular, then every operator field $M$ that commutes with $L$ (and in particular any of its symmetries) can be uniquely written as
\begin{equation}
\label{eq:bols2}
M = g_1 L^{n - 1} + \dots + g_n \operatorname{Id},
\end{equation}
where $g_i$ are some functions. We say that a symmetry $M$ is {\it regular} at a point $\mathsf p$, if the differentials $\ddd g_i$ are linearly independent at this point. In addition, we say that $M$ is centred at  $\mathsf p$, if $g_i (\mathsf p) = 0$, $i = 1, \dots, n$ (throughout the paper, $n$ always denotes the dimension of our manifold). 

\begin{Theorem}\label{t2_1}
Let $L$ be a $\gl$-regular Nijenhuis operator in a neighbourhood of $\mathsf p$. Then
\begin{enumerate}
    \item Every symmetry $M$ of $L$ is  strong.
    \item For any two symmetries $M$ and $R$, their product  $MR$ is also a symmetry.
    \item For any two symmetries $M$ and $R$,  one has $\langle M, R \rangle = 0$. In particular, every symmetry of $L$ is a Nijenhuis operator.
    \item The regular symmetries centred at $\mathsf p$ are in one-to-one correspondence with the systems of first companion coordinates centred at $\mathsf p$ in the sense that the coefficients $g_1,\dots, g_n$ of expansion \eqref{eq:bols2} are first companion coordinates for $L$ if and only if $M$ is a regular symmetry. In particular,  the existence of a regular symmetry  is equivalent to reducibility of $L$ to the first companion form \eqref{eq:comp1}.
\end{enumerate}
\end{Theorem}

\begin{Remark}{\rm The second statement of Theorem \ref{t2_1} implies that the symmetries of a $\gl$-regular operator $L$ form an algebra w.r.t. pointwise matrix multiplication.  In particular, if $M$ is a symmetry, then any polynomial  $p(M)$ with constant coefficients is a symmetry also.  Moreover, the same is true for any analytic function $f(M)$.
}\end{Remark}

We say that a $\gl$-regular Nijenhuis operator $L$ is in the {\it second companion form} if its matrix in local coordinates is
\begin{equation}
\label{eq:comp2}
L_{\mathsf{comp2}} = \left(\begin{array}{ccccc}
     0 & 1 & 0 & \dots & 0  \\
     0 & 0 & 1 & \dots & 0  \\
     &  &  & \ddots &   \\
     0 & 0 & 0 & \dots & 1 \\
     \sigma_n & \sigma_{n - 1} & \sigma_{n - 2} & \dots & \sigma_1 \\
\end{array}\right)
\end{equation}
The corresponding coordinates $x^1,\dots, x^n$ are called {\it second companion coordinates}. 
By \cite[Theorem 1.1]{nij3}, the operator $L_{\mathsf{comp2}}$ 
is Nijenhuis 
if and only if the 
 1-forms $\omega= \sigma_n\ddd x^1+\dots+ \sigma_1\ddd x^n$  and $L^*\omega$ are both closed.

It was shown in \cite[Theorem 1.1]{nij3} that in the real analytic category, any $\gl$-regular Nijenhuis operator can be put in the form \eqref{eq:comp2} near every point $\mathsf p$ by a suitable coordinate transformation. In the smooth category, combining the second statement of Theorem \ref{t3_2} with Theorems \ref{t1} and \ref{jordan},  one can construct second companion coordinates, if $\mathsf p$ is algebraically generic. The existence of such coordinates near singular points in the smooth case is unknown. As we shall see, this problem is closely related to the existence of regular conservation laws.

It is a well-known property of Nijenhuis operators (e.g., \cite{gm, ml}) that if $\ddd (L^* \ddd f) = 0$, then all the forms $(L^{*})^k \ddd f$ are closed too 
(because of its importance for our paper, we will prove it as Lemma \ref{c1}). This implies that (locally) there exist functions  $f=f_1, \dots, f_n$ such that $(L^*)^i \ddd f_1=f_{i+1}$ or, equivalently, $L^* \ddd f_i = \ddd f_{i + 1}$, $i = 1, \dots, n - 1$. The sequence of their differentials $\ddd f_1, \dots, \ddd f_n$ is called a {\it hierarchy of conservation laws}. 

We say that a conservation law $\ddd f$ is  {\it regular}, if the $1$-forms $(L^*)^i \ddd f$, $i = 0, \dots, n - 1$ are linearly independent. The corresponding hierarchy will be called a {\it regular hierarchy of conservation laws}. Note that the regularity condition is algebraic: it is equivalent to the fact that $\ddd f(\mathsf p)$ is a cyclic (co)vector for the operator $L^*: T^*_{\mathsf p}M\to T^*_{\mathsf p}M$ dual to $L$.

\begin{Theorem}\label{t3_2}
Let $L$ be a  $\gl$-regular Nijenhuis operator in a neighbourhood of $\mathsf p$. Then
\begin{enumerate}
    \item Every conservation law $\ddd f$ of $L$ is a conservation law for all of its symmetries, that is, $\ddd (M^* \ddd f) = 0$ for any symmetry $M$.
    \item The regular hierarchies of conservation laws $\ddd f_1, \dots, \ddd f_n$ are in one-to-one correspondence with the second companion coordinates, centred at $\mathsf p$  in the sense that  1-forms  $\ddd f_1, \dots, \ddd f_n$ form a regular hierarchy if and only if  $f_i-f_i(\mathsf p)$ are second companion coordinates centred at $\mathsf p$. In particular, the existence of regular hierarchies is equivalent to  the reducibility of $L$ to the second companion form \eqref{eq:comp2}.
\end{enumerate}
\end{Theorem}

Theorems \ref{t2_1} and \ref{t3_2} show that the symmetries and conservation laws of a $\gl$-regular Nijenhuis operator $L$ possess several remarkable properties:
\begin{itemize}
 \item[P1.]  Each symmetry of $L$ is strong.

\item[P3.]  Each symmetry of  $L$ is Nijenhuis.

\item[P2.]  If $M_1$ and $M_2$ are symmetries of  $L$, then their product $M_1 M_2$ is a symmetry also.

\item [P4.] Symmetries $M_1$ and $M_2$ commute is the algebraic sense, i.e.,  $M_1M_2 = M_2 M_1$, and are symmetries of each other.

\item[P5.]  Every conservation law $\ddd f$ of the operator $L$ is a conservation law for each of its symmetry $M$, that is, $\ddd ( M^*\ddd f)=0$.

\weg{\item[6.]  Let $\ddd f$ be a regular conservation law of  $L$.  Then any other conservation law  $\ddd g$ can be written in the form   $\ddd g =  M^* \ddd f$ where $M$ is a suitable symmetry of $L$. }
\end{itemize}

The next two examples show that for Nijenhuis operators that are not $\gl$-regular, these properties may not hold.

\begin{Ex}{\rm
Consider the constant operator $L= \begin{pmatrix}  0 & 1 & 0 \\ 0 & 0 & 0 \\  0 & 0 & 0 \end{pmatrix}$ in $\R^3(x,y,z)$, which consists of two nilpotent Jordan blocks of size 2 and 1. The symmetries of $L$ have the following form
$$
M = \begin{pmatrix} 
f &  x f_y + g & x f_z + a \\
0 & f & 0\\
0 & b & c
\end{pmatrix} ,$$
where the functions $f, g, a,b,c$ depend on $y$ and $z$ only. Strong symmetries have a similar form with the additional condition that $f = f(y)$ (i.e., $f$ does not depend on $z$). The conservation laws are $\ddd(x u(y) + v(y,z))$.

None of properties P1 -- P5 are met.
}\end{Ex}

\begin{Ex}{\rm

Consider the operator $L= \begin{pmatrix}  z & 0 & 0 \\ 0 & z & 1 \\  0 & 0 & z \end{pmatrix}$  in $\R^3(x,y,z)$. This is a local normal form,  in $\dim = 3$,
for a Nijenhuis operator, whose Jordan normal form consists of two blocks with the same non-constant eigenvalue $\lambda$,   $\ddd\lambda \ne 0$.

The symmetries of $L$ have the form
$$
M = \begin{pmatrix} 
f  +  a \, e^{-y} &  0   &  b\, e^{-y} \\
f_x + c \, e^{-y} & f & f_z + g \, e^{-y}\\
0 & 0 & f
\end{pmatrix},
$$
where the functions $f, g, a, b, c$ depend on $x$ and $z$ only. 
The strong symmetries are distinguished by the additional condition $a=c=0$.   The conservation laws of $L$ have the form $\ddd (a(x,z)e^y + b(z))$.

None of properties P1 -- P5 are met.

}\end{Ex}

Based on the above results, we can now give a complete description of conservation laws and symmetries of $\gl$-regular operators at algebraically generic points.  Of course, this result is well known (and directly follows from Theorem \ref{t1}) for diagonal operators 
$L = \operatorname{diag}(\lambda_1(u^1), \dots, \lambda_n(u^n))$.  In this case,  the symmetries and conservation laws are of the form $M =  \operatorname{diag}(m_1(u^1), \dots, m_n(u^n))$ and $\ddd f(u) = \ddd f_1(u^1) + \dots + \ddd f_n(u^n)$.   If a $\gl$-regular operator $L$ contains Jordan blocks, then due to the splitting theorem, the problem reduces to a description of symmetries and conservation laws for a single Jordan block.  For our purposes, it is convenient to deal with  the Nijenhuis operators in the Toeplitz form
\begin{equation}
\label{eq:nilpot}
N =  \left(\begin{array}{ccccc}
     0 & 1 & 0 & \dots & 0  \\
     0 & 0 & 1  & \dots & 0 \\
     \vdots & \vdots & \ddots & \ddots & \\
      0 & 0 & \dots &  0 & 1  \\
     0 & 0 & \dots & 0 & 0  \\
\end{array}\right), \quad U = \left(\begin{array}{ccccc}
     u^n & u^{n - 1} & u^{n - 2} & \dots & u^1  \\
     0 & u^{n} & u^{n - 1} & \dots & u^2  \\
     0 & 0 & u^{n} & \dots & u^3  \\
     \vdots & \vdots & \ddots & \ddots & \vdots \\
     0 & 0 & \dots &0 & u^n  \\
\end{array}\right).
\end{equation}

We notice that Toeplitz matrix $U$ with $u^1, \dots, u^n$ being local coordinates defines a Nijenhuis operator and moreover every Nijenhuis operator which is similar to a Jordan block with a non-constant eigenvalue $\lambda$, with $\ddd\lambda \ne 0$, can be reduced to the Toeplitz normal form $U$  (with $u^{n-1}\ne 0$).  There are other canonical forms for Nijenhuis Jordan blocks (see \cite{nij, gm, turiel}), but this one is the most convenient for our purposes. 

For an arbitrary smooth function $h$, the operator field $h(U)$ is defined as
\begin{equation}\label{hf}
    h(U) =  \left(\begin{array}{ccccc}
     h_n(u^n) & h_{n - 1} (u^n, u^{n - 1}) & h_{n - 2}(u^n, u^{n - 1}, u^{n - 2}) & \dots & h_1 (u^n, \dots, u^1)  \\
     0 & h_n(u^{n}) & h_{n - 1} (u^n, u^{n - 1}) & \dots & h_2 (u^n, \dots u^2)  \\
     0 & 0 & h_n(u^{n}) & \dots & h_{3}(u^n, \dots u^3)  \\
     \vdots & \vdots & \ddots & \ddots & \vdots \\
     0 & 0 & \dots &0 & h_n(u^n)  \\
\end{array}\right),
\end{equation}
where the functions $h_i$, $i = 1, \dots, n$ are defined from the expansion
$$
h\big(u^n + \lambda u^{n - 1} + \dots + \lambda^{n - 1} u^1\big) \simeq h_n + \lambda h_{n - 1} + \dots + \lambda^{n -1} h_1 + \{\text{terms with $\lambda^i$ for $i \geq n$}\}.
$$

Without loss of generality, we assume that $U$ is given in a neighbourhood of the point $(u^1,\dots, u^n)=(0,\dots,0)$. 

\begin{Theorem}\label{jordan}
Every symmetry of the Nijenhuis operator $U$ is a symmetry of $N$ and vice versa. These symmetries $M$ are parametrised by $n$ arbitrary smooth functions $f_1, f_2, \dots , f_n$ of one variable defined in a neighbourhood of $0\in\R$ and take the form
$$
M = f_1 (U) N^{n-1} + f_2 (U) N^{n-2} + \dots + f_n(U).
$$

Similarly, each conservation law of $U$ is a conservation law of $N$ and vice versa.  Let $f(u) = M^1_n$ be the element of $M$ in the right upper corner.  Then $\ddd f$ is a conservation law, and every conservation law can be obtained in this way.
\end{Theorem}

If $L$ is conjugate to a pair of $n\times n$ Jordan blocks with complex eigenvalues $\lambda$ and $\bar\lambda$, then the description of symmetries and conservation laws is quite similar.  In view of Theorem 3.3 from \cite{nij},   $L$ induces a canonical complex structure $J$ and then in suitable complex coordinates $u^1, \dots, u^n$, $u^\alpha = v^\alpha + \mathrm{i} w^\alpha$,  the operator $L$ is $J$-linear and can be written as a complex $n\times n$ matrix of the same form as $U$.  The statement of Theorem \ref{jordan} remains the same with the only amendment that $f_i$ are holomorphic in a neighbourhood of $0\in\mathbb C$.

Combining    Theorem \ref{jordan} with Theorem \ref{t1} leads to a complete description of conservation laws and symmetries for $\gl$-regular Nijenhuis operators at almost every point $\mathsf p\in M$.  Namely,  it suffices to assume that in a neighborhood of $\mathsf p$, each eigenvalue $\lambda$ of $L$ has constant multiplicity and is either constant or $\ddd \lambda \ne 0$ (both cases are covered by Theorem \ref{jordan}).

Before proceeding to the next Theorem, let us recall the definition of an analytic function of an operator field. Consider a real analytic function $v(z)$ in a neighbourhood of $z = 0\in \R$. Its Taylor series $v(z) = a_0 + a_1 z + a_2 z^2 + \dots$ converges if $|z| < r$ for some $r>0$. Now assume that in some coordinates $u^1, \dots, u^n$ centred at $\mathsf p$ we have an operator field $U$ analytic in $u^1, \dots, u^n$ and such that $U(\mathsf p) = 0$. Then the series $v(U) = a_0 \operatorname{Id} + a_1 U + a_2 U^2 +\dots$ converges for $||U|| < r$, where $|| U || = \max_{\xi\ne 0} \frac{|U(\xi)|}{|\xi|}$. Thus, any analytic function $v(z) = a_0 + a_1 z + a_2 z^2 + \dots$ of one variable defines the operator field $v(U)$ in a neighbourhood of $\mathsf p$. \weg{We shall later use this notion without explaining the details.} The next theorem treats the real analytic case. It shows that in a neighborhood of any point $\mathsf p$, regardless of whether it is algebraically generic or singular,  the symmetries are naturally parametrised by $n$ functions of one variables.  Moreover this parametrisation is explicit as soon as we know at least one regular symmetry.  A similar statement holds for conservation laws.

\begin{Theorem}\label{t6}
Let $L$ be a real analytic Nijenhuis operator, $\gl$-regular at a point $\mathsf p$. Then 
\begin{enumerate}
    \item There exists a regular symmetry $U$ centred at $\mathsf p$, and a regular conservation law $\ddd f$.
    \item For any collection of functions  $v_i$ analytic in a neighbourhood of $0\in\R$, and any regular symmetry $U$ centred at $\mathsf p$, the formula
    \begin{equation}\label{s2}
    M = v_1(U) L^{n - 1} + \dots + v_{n - 1}(U) L + v_n(U)
    \end{equation}
    defines a symmetry. Moreover, given a regular symmetry $U$ centred at $\mathsf p$, every symmetry can be written in the form \eqref{s2} with an appropriate choice of functions $v_i$.
    \item Given a regular conservation law $\ddd f$, for any conservation law $\ddd h$ there exists a symmetry $M$ such that $\ddd h = M^* \ddd f$.
\end{enumerate}    
\end{Theorem}

\begin{Remark}{\rm
We do not know if the statement of this theorem is still true in the smooth category.  In our opinion, this is an important open problem in Nijenhuis geometry. The major issue here is that we are not aware of a suitable analog of the Cauchy-Kovalevskaya theorem, which could guarantee, in the smooth category, the existence of good solutions for the PDE systems responsible for symmetries and conservation laws at singular points of the operator $L$.  At algebraically generic points,  Theorem \ref{t6} holds in the smooth case due to the splitting Theorem \ref{t1} and the description of symmetries and conservation laws for Jordan blocks from Theorem \ref{jordan}.

}\end{Remark}

\subsection{Main Results: Applications}

Given a Nijenhuis operator $L$, we consider the following quasilinear system  of PDEs of hydrodynamic type on $n$ unknown functions  $u^1(x,t_1,...,t_{n-1}), ... , u^n(x, t_{1},...,t_{n-1})$ of $n$  variables:
\begin{equation}\label{sys}
u_{t_{i}} = A_i u_x, \quad i = 1, \dots, n - 1,     
\end{equation}
where the operator fields $A_i$, are defined from the following recursion relations
\begin{equation}\label{rec}
A_0 = \operatorname{Id}, \quad A_{i + 1} = L A_i - \sigma_i \operatorname{Id}, \quad i = 0, \dots, n - 1.    
\end{equation}
Equivalently, the operators $A_i$ can be defined from the matrix relation
$$
\det (\lambda\,\Id-L) (\lambda\,\Id - L)^{-1} =  \lambda^{n-1} A_0 +  \lambda^{n-2} A_1 +\dots + \lambda A_{n-2} + A_{n-1}.
$$

The system \eqref{sys} can be obtained within the framework of the general construction introduced by Magri and Lorenzoni in \cite{ml, m}. 
In particular, the equations in system \eqref{sys} are compatible,  which implies that (in the real-analytic cathegory)  for any initial curve $\gamma(x)$ there exists a solution $u(x,t_1,..,t_{n-1})$ such that $u(x,0,...,0)= \gamma(x)$.
In the case of diagonal Nijenhuis operator $L = \operatorname{diag} (u^1, \dots, u^n)$, the corresponding system satisfies the semi-hamiltonicity condition of Tsarev \cite{Tsarev} and is weakly-nonlinear in the sense of Rozhdestvenskii. Such systems were studied and completely integrated by Ferapontov  \cite{fer, fer1}  in quadratures (we explain the exact meaning of this term below).

The generalised hodograph method of Tsarev \cite{Tsarev} uses symmetries to integrate semi-hamiltonian systems of hydrodynamtic type. Conservation laws are also closely related to the integrability of such systems. The next theorem gives a description of symmetries and conservation laws for system \eqref{sys}, by reducing this problem to the case of a Nijenhuis operator, in which such a description  at almost every point is given by Theorems \ref{t1} and \ref{jordan}.

\begin{Theorem}\label{t4}
Let $L$ be a $\gl$-regular Nijenhuis operator in a neighbourhood of $\mathsf p$ and $A_i$ be the operator fields from  \eqref{rec}. Then the following statements hold:
\begin{enumerate}
    \item For any hierarchy of conservation laws $\ddd f_1, \dots, \ddd f_n$ of $L$, the operator field
    $$
    B = f_1 A_n + \dots + f_n A_1
    $$
    is a common symmetry for $A_i$. Moreover, every common symmetry in a neighbourhood of $\mathsf p$ can be written in this way with an appropriate choice of the hierarchy $f_1, \dots, f_n$.
    \item For any symmetry $M = g_1 L^{n - 1} + \dots + g_n \operatorname{Id}$ of $L$,  the differential $\ddd g_1$ of the first function $g_1$ is a common conservation law of $A_i$. In particular
    $$
    A_i \ddd g_1 = \ddd g_{i + 1}.
    $$
    Moreover, every common conservation law of $A_i$'s in a neighbourhood of $\mathsf p$ can be obtained in this way with an appropriate choice of a symmetry $M$.
\end{enumerate}
\end{Theorem}

Following the classics, we say that a system of ODEs or PDEs {\it is integrable in quadratures} if for (almost) all initial conditions the corresponding solution can be found by solving systems of functional equations and integrating closed $1$-forms. The next theorem is crucial for solving  \eqref{sys} in quadratures. 

\begin{Theorem}\label{t5}
Consider the PDE system \eqref{sys} with a $\gl$-regular Nijenhuis operator $L$ in a neighbourhood of $\mathsf p$. Choose an initial condition $u(0,\dots, 0, x) = \gamma (x)$ such that $\gamma'(x)$ is a cyclic vector for $L$ and $\gamma(0) = \mathsf p$.  Assume there exists a regular hierarchy $\ddd g_1(u),\dots, \ddd g_n(u)$ of conservation laws of $L$ such that $g_1(\gamma(x)) = x$ and 
$g_i(\gamma(x)) = 0$ for $i=2,\dots, n$. Then
 
\begin{itemize}
\item The solution of  \eqref{sys}  with the prescribed initial condition can be obtained by solving the system of equations
$$
\begin{aligned}
g_1(u^1, &\dots, u^n) = t_{n - 1},\\
g_2(u^1, &\dots, u^n) = t_{n - 2},\\
& \dots \\
g_{n - 1}(u^1, &\dots, u^n) = t_1,\\
g_n (u^1, &\dots, u^n) = x,\\
\end{aligned}
$$
with respect to $u^1,\dots, u^n$.
\item The required hierarchy $g_1,\dots, g_n$ can be obtained by the following procedure:
\begin{itemize}
\item[Step 1.] Take an arbitrary  regular hierarchy $\ddd f_1, \dots, \ddd f_n$ of conservation laws of $L$;
 \item[Step 2.] At every point of the curve $\gamma(x)$,  find  the operator $\widehat M(\gamma(x))$ that commutes with $L(\gamma(x))$ and satisfies the relation
     \begin{equation}
     \label{eq:from1.7_a}
    \widehat M(\gamma(x)) \gamma'(x) = \xi(x),  
     \end{equation}
     where $\xi(x)$ is uniquely defined by the condition 
     \begin{equation}
     \label{eq:from1.7_b}
     \bigl( \ddd f_i \bigl(\gamma(x)\bigr) , \xi(x)    \bigr) = \begin{cases} 0, &\mbox{if } i=1,\dots,n-1\\ 1, &\mbox{if } i=n. \end{cases}.
     \end{equation}
\item[Step 3.]  Find the symmetry $M$ of $L$ such that restriction $M$ onto the curve $\gamma(x)$ coincides with $\widehat M(\gamma(x))$.  
     
\item[Step 4.]  Find $g_i$  by integrating the closed $1$-forms $M^* \ddd f_i$, i.e. by solving  $\ddd g_i=M^* \ddd f_i$.

\end{itemize}
\end{itemize}
\end{Theorem}

\begin{Remark}{\rm

Let us explain the key observations leading to the algorithm given in  Theorem  1.7 (the proof is different and is based on a linear algebra trick; note also that in view of Theorem \ref{t6} that describes conservation laws for Nijenhuis operators via symmetries, Theorem \ref{t5} can be obtained from Theorem \ref{t4} using  a natural modification of Tsarev's hodograph method for $\gl$-regular operators). The first observation is that every second companion coordinate system provides us with a particular solution to system \eqref{sys} (see Lemma \ref{im1} below). Next, the second companion coordinates can be constructed starting from a regular conservation law. Finally,  all conservation laws of $L$ can be obtained from one single regular conservation law $\ddd f$ by applying suitable symmetries $M$ to it. Thus, if we have a `sufficiently large' family of symmetries of $L$ (say, a family parameterised by $n$ functions of one variable) and one single regular conservation law, then we can construct a sufficiently large family of  solutions of \eqref{sys}.

Thus, the first potential issue for realisation of the algorithm from Theorem \ref{t5} is the existence problem for large families of regular symmetries and conservation laws. Theorems \ref{t2_1} and \ref{t3_2} imply that this problem is equivalent to the existence problem for  companion coordinate systems. As discussed earlier, the latter problem is resolved for algebraically generic points of $\gl$-regular Nijenhuis operators $L$ in the smooth category (this follows from Theorem \ref{t1}) and for all points, including singular,  in the real analytic category (this follows from Theorem \ref{t6}).  The existence problem remains open for singular point in the smooth case. 
}\end{Remark}

It is not obvious that, even though the necessary objects exist, one can perform  {\it in quadratures} all the steps listed in Theorem  \ref{t5}. This concerns, first of all, the construction of a symmetry $M$ with prescribed initial condition on the initial curve $\gamma$. The examples below show that this can be done for algebraically generic points $\mathsf p$ (with the help of formula \eqref{hf} and Theorem \ref{jordan}){\weg{and also for singular $\mathsf p$ under the deferential non-degeneracy condition}}.


\subsection{Two examples} 

Let us give two examples of how    Theorem \ref{t5} works. Both examples are 4-dimensional  but of course the method works similarly 
in all dimensions. In the first example, the operator $L$ is diagonal.
 As  mentioned  before, the diagonal case is  well understood by E. Ferapontov  \cite{fer, fer1}, see also \cite[Theorem 4]{MB2010};  
experts will immediately see that  \eqref{sys:fer} is the 4-dimensional version of the system obtained in 
 \cite{fer1, MB2010}. In the second example, the operator $L$ is blockdiagonal with two blocs of size $2\times 2$. In both examples, we consider a smooth curve $\gamma:[-\varepsilon, \varepsilon]\to \mathbb{R}^4(u^1,u^2, u^3, u^4)$, $\gamma(t)=\bigl(\gamma_1(t), \gamma_2(t),\gamma_3(t),\gamma_4(t)\bigr)$, 
and  work in a sufficiently small neighbourhood of $\gamma(0)$.  We also  assume  that  $L$ is $\gl$-regular at the point $\gamma(0)$ which means that the 4 numbers $\gamma_1(0), \gamma_2(0), \gamma_3(0), \gamma_4(0)$ are mutually different in the first example, and $\gamma_2(0) \ne  \gamma_4(0)$, $\gamma_1(0)\ne 0$, $\gamma_3(0)\ne 0$ in the second example.  
 
\begin{Ex}  \label{ex:1}{\rm  Suppose $L= \operatorname{diag}(u^1, u^2, u^3, u^4)$.   
Then, by Theorem \ref{t1}, every symmetry of $L$ has the form $M= \operatorname{diag}(F_1(u^1), F_2(u^2), 
F_3(u^3), F_4(u^4))$ and the differentials of the functions 
\begin{eqnarray*}
f_1 &=& u^1+u^2+u^3 +u^4,  \\ 
f_2 &=& \tfrac{1}{2}\left((u^1)^2+(u^2)^2+(u^3)^2 +(u^4)^2\right), \\ 
f_3 &=& \tfrac{1}{3}\left((u^1)^3+(u^2)^3+(u^3)^3 +(u^4)^3\right),\\ 
f_4 &=& \tfrac{1}{4}\left((u^1)^4+(u^2)^4+(u^3)^4 +(u^4)^4\right)
\end{eqnarray*}
 form a regular hierarchy. 
Next, let $\gamma(x)= (\gamma_1(x),..., \gamma_4(x))$ be a curve such that its velocity vector $\gamma'= (\gamma'_1,..., \gamma'_4)$ is cyclic for every $x$, which means $\gamma'_i\ne 0$
for every $i$. The vector $\xi(x)$ from \eqref{eq:from1.7_b} is then given by   
$$
\xi= 
 \left(\tfrac {1}{ \left( \gamma_1  -\gamma_4 \right) 
 \left( \gamma_1  -\gamma_3   \right) 
 \left( \gamma_1  -\gamma_2   \right) }\, 
, \ \tfrac {1}{ \left( \gamma_2 -\gamma_1  \right)  \left( \gamma_2  -\gamma_{{4}} \right) \left( \gamma_2  -\gamma_3 \right) }\, 
,\
 \tfrac{1}{ \left( \gamma_3  -\gamma_2
   \right)  \left( \gamma_3  -\gamma_1
   \right)  \left( \gamma_3  -\gamma_{
4} \right) }\, 
, \ \tfrac {1}{ \left( \gamma_4  -\gamma_{{
2}} \right)  \left( \gamma_4 -\gamma_3 \right) 
 \left( \gamma_4 -\gamma_{1} \right) }\right). 
$$
In view of \eqref{eq:from1.7_a},  the symmetry $M$ on the curve $\gamma(x)$ is given by 
$$\widehat M = \operatorname{diag}
 \left(\tfrac {1}{ \gamma'_1\left( \gamma_1  -\gamma_4 \right) 
 \left( \gamma_1  -\gamma_3   \right) 
 \left( \gamma_1  -\gamma_2   \right) }\, 
, \ \tfrac {1}{\gamma'_2 \left( \gamma_2 -\gamma_1  \right)  \left( \gamma_2  -\gamma_{{4}} \right) \left( \gamma_2  -\gamma_3 \right) }\, 
,\
 \tfrac{1}{ \gamma'_3\left( \gamma_3  -\gamma_2
   \right)  \left( \gamma_3  -\gamma_1
   \right)  \left( \gamma_3  -\gamma_{
4} \right) }\, 
, \ \tfrac {1}{ \gamma'_4\left( \gamma_4  -\gamma_{{
2}} \right)  \left( \gamma_4 -\gamma_3 \right) 
 \left( \gamma_4 -\gamma_{1} \right) }\right). 
$$
If $\gamma_i^{-1}$ denotes the inverse function to the function $\gamma_i(x)$, which exists because $\gamma'_i\ne 0$, then the  
$i$-th  diagonal element of $M$ is given by 
$$F_i(u^i)= \frac {1}{ \gamma'_i (\gamma_i^{-1}(u^i))\prod_{j\ne i}\left(u^i   -\gamma_j(\gamma_i^{-1}(u^i)) \right)}.$$ 
 
Hence, the  form $\ddd g_i$ is   and the functions $g_i$ are given by 
\begin{eqnarray*}
\ddd g_i &=&  F_1(u^1)(u^1)^{i-1} \ddd u^1 +   F_2(u^2)(u^2)^{i-1}  \ddd u^2+  F_3(u^3)(u^3)^{i-1}  \ddd u^3 +  F_4(u^4)(u^4)^{i-1}  \ddd u^4,\\
	g_i(u) &=&\int_{\gamma_1(0)}^{u^1}   F_1(s)s^{i-1}  \ddd s+ \int_{\gamma_2(0)}^{u^2}   F_2(s)s^{i-1}  \ddd s+ \int_{\gamma_3(0)}^{u^3}   F_3(s)s^{i-1}  \ddd s+\int_{\gamma_4(0)}^{u^4}   F_4(s)s^{i-1}  \ddd s.
	\end{eqnarray*}
  	 The solution of \eqref{sys} with  the initial condition $u(x,0,0,0)= \gamma(x)$ can then be obtained by resolving the system of algebraic relations 
\begin{equation} 
\label{sys:fer} 
\begin{aligned}
g_1(u) &= t_3,\\
g_2(u) &= t_2,\\
g_3(u) &= t_1,\\
g_4(u)&= x  
\end{aligned}
\end{equation} 
}\end{Ex}

\begin{Ex} \label{ex:2}{\rm  Suppose our  4-dimensional $\gl$-regular Nijenhuis  operator $L$ has  
two nonconstant (= the differential in never zero) 
 eigenvalues of geometric multiplicity 2; by Theorem \ref{t1}  in some coordinate system it is given by 
$$
L= \left(\begin {array}{cccc} u^{{2}}&u^{{1}}&0&0\\ \noalign{\medskip}0
&u^{{2}}&0&0\\ \noalign{\medskip}0&0&u^{{4}}&u^{{3}}
\\ \noalign{\medskip}0&0&0&u^{{4}}\end {array} \right). 
$$
By Theorem  \ref{t1},  in
 order to describe its symmetries and find  a regular conservation law,  it suffices to do this for the two-dimensional Nijenhuis operator  
$\left(\begin{array}{cc} u^2 & u^1 \\ 0 & u^2\end{array}\right)$. This is done in Theorem \ref{jordan}: every symmetry has the form 
$$M= \left( \begin {array}{cccc} F_1 ( u^{{2}}) &
 { F'_1} ( u^{{2}})  u^{{1}}+{F_3} ( u^{{2}} ) &0&0
\\ \noalign{\medskip}0&{ F_1} ( u^{{2}} ) &0&0
\\ \noalign{\medskip}0&0&{F_2}( u^{{4}} ) & {F'_2} ( u^{{4}}  
 ) u^{{3}}+{F_4} ( u^{{4}} ) 
\\ \noalign{\medskip}0&0&0&{F^2} ( u^{4}) 
\end {array} \right)
$$ 
and a regular hierarchy of conservation laws 
is formed by differentials of the functions 
\begin{eqnarray*}f_1 &=& u^{{1}}+u^{{3}}\\ 
f_2 &=& u^{{2}}u^{{1}}+u^{{4}}u^{{3}}\\ f_3 &=& (u^{{2}})^{2}u^{{1}}+(u^{{4}})^{2}u^{{3}}\\ f_4&=& (u^{{2}})^{3}u^{{1}}+(u^{{4}})^{3}u^{{3}}\end{eqnarray*}
Next, let $\gamma(x)= (\gamma_1(x),..., \gamma_4(x))$ be a curve such that for every $x$ the vector $\gamma'= (\gamma'_1,..., \gamma'_4)$ is cyclic, which in this case means $\gamma'_2\ne 0$ and $\gamma'_4\ne 0$. 
The vector $\xi(x)$ from \eqref{eq:from1.7_b} is then given by   
$$\xi= 
\left(\frac{2}{\left( \gamma_4   -\gamma_2  
 \right)^{3}},\frac {1}{\gamma_1   \left( \gamma_4
  -\gamma_2   \right)^{2}},\frac{2}{
 \left( \gamma_2  -\gamma_4   \right)^{3}},{\frac {1}{\gamma_3   \left( \gamma_2  -\gamma_4   \right) ^{2}}}\right), $$
and  the equation $\widehat M(\gamma(x)) \gamma'(x) = \xi(x)$ 
reads 
\begin{equation} \begin{array}{cc}   F_1 \left( \gamma_2 \right) 
\gamma'_2  =\tfrac {1}{\gamma_1   \left( 
\gamma_2  -\gamma_4   \right) ^{2}}\ ,  \ &  
	F_1 \left( \gamma_2 \right) \gamma'_1 + 
  \left( F_1'
 \left( \gamma_2 \right)   \gamma_{1}+F_3 \left( \gamma_2
 \right)  \right) \gamma'_2  =\tfrac{2}{
\left( \gamma_4  -\gamma_2  
 \right) ^{3}}, \\ 
F_2 \left( \gamma_4 \right) \gamma'_4
  =\tfrac {1}{\gamma_3   \left( 
	\gamma_4  -\gamma_2   \right)^{2}} \ ,  \  &   
F_2\left( \gamma_4 \right) \gamma'_3  +   \left( {F_2'} \left( \gamma_4 \right)   \gamma_3+{
 F_4} \left( \gamma_4 \right) \right)  
\gamma'_4 =\tfrac{2}{\left( \gamma_2 -
\gamma_4   \right)^{3}}. \end{array} \label{4eq} \end{equation}
The first and third equations  of \eqref{4eq} 
give us the formula for $F_1(u^2)$ and   $F_2(u^4)$: 
\begin{eqnarray}
F_1 \left( u^2 \right) 
 &=&\tfrac {1}{\gamma_1(\gamma_2^{-1}(u^2))   \left( 
u^2  -\gamma_4(\gamma_2^{-1}(u^2)   \right) ^{2} \gamma'_2(\gamma_2^{-1}(u^2)) }, \label{eq:F1} \\
 F_2 \left( u^4 \right) 
 &=&\tfrac {1}{\gamma_3(\gamma_4^{-1}(u^4))   \left( 
u^4  -\gamma_2(\gamma_4^{-1}(u^4)   \right) ^{2} \gamma'_4(\gamma_4^{-1}(u^4)) }, \label{eq:F2} 
\end{eqnarray}
where $\gamma_2^{-1}$ and  $\gamma_4^{-1}$ denote the inverse functions of $\gamma_2$ and $\gamma_4$.

The second and fourth equations of  \eqref{4eq}, after substituting $\gamma^{-1}_2$ and $\gamma^{-1}_4$, have the form  
\begin{eqnarray*}F_1 \left( u^2 \right) \gamma'_1(\gamma^{-1}_2(u^2)) + 
  \left(F_1'
 \left( u^2 \right)   \gamma_{1}(\gamma_2^{-1}(u^2))+F_3 \left( u^2\right)  \right) \gamma'_2(\gamma_2^{-1}(u^2))  &=&\tfrac{2}{
\left( \gamma_4(\gamma_2^{-1}(u^2))  -u^2  
 \right) ^{3}} \\
 F_2\left( u^4  \right) \gamma'_3(\gamma_4^{-1}(u^4))  +   \left( {F_2'} \left( u^4 \right)  \gamma_3(\gamma_4^{-1}(u^4))+{
 F_4} \left( u^4 \right) \right)  
\gamma'_4(\gamma^{-1}_{4}(u^4))  &=&\tfrac{2}{\left( \gamma_2(\gamma^{-1}_4(u^4))  -
u^4   \right)^{3}}.
\end{eqnarray*}
Combining them with \eqref{eq:F1} and \eqref{eq:F2} gives us  the formula for $F_3(u^2)$ $F_4(u^4)$.  Plugging $F_1, F_2, F_3, F_4$ into the formula for $M$, we obtain closed forms $\ddd g_i$; their integration gives the functions $g_i$ (the constant of integration should be chosen so that $g_1(\gamma(0))= g_2(\gamma(0)) =
g_3(\gamma(0))=g_4(\gamma(0))=0$). Finally, we obtain the solution of \eqref{sys} with  the initial condition $u(x,0,0,0)= \gamma(x)$ by solving the system of algebraic relations
$$
\begin{aligned}
g_1(u)& = t_{3},\\
g_2(u) &= t_{2},\\
g_{3}(u) &= t_1,\\
g_{4}(u) &= x\\
\end{aligned}
$$
with respect to $u=(u^1, u^2, u^3, u^4)$. 
}\end{Ex} 


\section{Proofs: General Theory}

\subsection{Splitting theorem for symmetries and conservation laws:  proof of Theorem \ref{t1}}

The first statement of Theorem \ref{t1} is the Splitting Theorem for Nijenhuis operators \cite[Theorem 3.1]{nij}, see also \cite{gm, turiel}. So we assume that in a coordinate system $\underbrace{u^1_1, \dots, u^{m_1}_1}_{u_1}, \underbrace{u^1_2, \dots, u^{m_2}_2}_{u_2}$,  the Nijenhuis operator $L$ splits into two Nijenhuis blocks, i.e., $L(u_1, u_2) = \begin{pmatrix} L_1(u_1) & 0 \\ 0 & L_2(u_2)\end{pmatrix}$ and the characteristic polynomials of $L_1$ and $L_2$ have no common roots. 

We start with the following algebraic fact.

\begin{Lemma}\label{lem:algebraic}
Consider three matrices $L_1\in \mathrm{Mat}(n\times n), R\in \mathrm{Mat}(n\times m)$ and $ L_2\in \mathrm{Mat}(m\times m)$ such that   
\begin{equation}\label{eq:1} 
L_1 R = RL_2 .
\end{equation}
If  $\operatorname{spectrum}(L_1)\cap \operatorname{spectrum}(L_2)= \varnothing$, then $R=0$.
\end{Lemma}

\begin{proof}   The relation $L_1 R = RL_2$ obviously implies  $L_1^k R = R L_2^k$  for any $k\in \mathbb N$ and more generally
$p(L_1)  R = R \, p(L_2)$ for any polynomial.    Let  $p(t)=p_{L_1}(t)$ be the characteristic polynomial of $L_1$, then the left hand side of the latter relation vanishes.  Hence $R \, p_{L_1}(L_2)=0$.  Since $L_1$ and $L_2$ have disjoint spectra, the matrix $p_{L_1}(L_2)$ is invertible, implying $R=0$.    
\end{proof} 

\begin{Remark}  \label{rem:transposed}  {\rm Lemma \ref{lem:algebraic} remains true if we replace  $L_2$ with its transpose matrix $L_2^\top$.   Indeed, the spectra of $L_2$ and  $L_2^\top$ coincide.
}\end{Remark}

Let us now prove the second statement of Theorem \ref{t1}. If $\ddd f$ is  a conservation law, then  the two-form $\ddd(L^*df) $  evaluated on the pair of vectors  $ \left( \tfrac{\partial }{\partial u^i_1}, \tfrac{\partial }{\partial u^j_2} \right)$  is zero. This gives    
\begin{equation} 
\label{eq:6} 
 \sum_{s=1}^{m_1}\left( \left(\tfrac{\partial }{\partial u_2^j} \overset{1}L{}^s_i \right)   \tfrac{\partial f}{\partial u^s_1}+ \overset{1}L{}^s_i\tfrac{\partial^2  f}{\partial u_2^j \partial u_1^s}\right) = 
\sum_{r=1}^{m_2}\left(\left(\tfrac{\partial }{\partial u_1^i} \overset{2}L{}^r_j \right)   \tfrac{\partial f}{\partial u_2^r} + \overset{2}L{}^r_j\tfrac{\partial^2  f}{\partial u_2^r \partial u_1^i}\right). 
 \end{equation}  
We have $\left(\tfrac{\partial }{\partial u_2^j} \overset{1}L{}^s_i \right) =0$ and $\left(\tfrac{\partial }{\partial u_1^i} \overset{2}L{}^r_j \right)=0$. Then, 
 \eqref{eq:6} has the form \eqref{eq:1} with $R= \left( \tfrac{\partial^2  f}{\partial u_1^i  \partial u_2^j}  \right)$.
By Lemma  \ref{lem:algebraic} and  Remark \ref{rem:transposed}, $R=0$ which implies that   $f=f_1(u_1)+ f_2(u_2)$ as  claimed. The equation  $\ddd (L^* \ddd f) = 0$ decomposes then into two equations $\ddd(L_1^*\ddd f_1) =0$ and $\ddd(L_2^*\ddd f_2) =0$ implying the second part of the second statement of Theorem \ref{t1}.

To prove the third statement of Theorem \ref{t1}, we write the  symmetric part of $\langle L, M \rangle$ acting on vector fields $\xi, \eta$ as follows:
\begin{equation}\label{eq:7}
\begin{aligned}
 & M[L\xi, \eta] + L[\xi, M\eta] - [L\xi, M\eta] - LM[\xi, \eta] + \\
 + &  M[L\eta, \xi] + L[\eta, M\xi] - [L\eta, M\xi] - LM[\eta, \xi]   =0.
\end{aligned} 
\end{equation}
Because $L$ and $M$ commute, the operator $M$ is blockdiagonal, $M= \operatorname{diag}(M_1, M_2)$. It is sufficient to prove the statement for one of the blocks, we will do it  for $M_1$. Take  $\xi=\tfrac{\partial  }{\partial u^i_1}$ and $\eta=\tfrac{\partial  }{\partial u_2^j}$. Then the first, fourth, fifth and the last terms of  \eqref{eq:7} disappear and   we obtain 
 \begin{equation} \label{eq:8} L[\xi, M\eta] - [L\xi, M\eta]+ L[\eta, M\xi] - [L\eta, M\xi]=0.  \end{equation} 
For  $k\in  \{1,...,m_1\}$, the $k$-th component of \eqref{eq:8} gives 
$$
\sum_{\alpha= 1}^n \overset{1}L{}^k_\alpha \tfrac{\partial }{\partial u_2^j} \overset{1}M{}^\alpha_i= \sum^{m_2}_{\beta=1}\overset{2}L{}^\beta_j \tfrac{\partial}{\partial u_2^\beta} \overset{1}M{}^k_i. 
$$
By Lemma \ref{lem:algebraic} applied, for every fixed $i\in \{1,...,m_1\}$,  to  the $m_1\times m_2$ matrix $R^\alpha_j:=\tfrac{\partial }{\partial u_2^j} \overset{1}M{}^\alpha_i$,     we obtain  $\tfrac{\partial }{\partial u_2^j} \overset{1}M{}^\alpha_i=0$. Thus, the components of $M_1$ do not depend on the variables $u_2^1,...,u_2^{m_2}$. Substituting now $\xi= \tfrac{\partial }{\partial u_1^i}$ and $\eta= \tfrac{\partial }{\partial u_2^j}$ in \eqref{eq:7}, we obtain 
$$
\begin{aligned}
    & M_1[L_1\xi, \eta] + L_1[\xi, M_1\eta] - [L_1\xi, M_1\eta] - L_1M_1[\xi, \eta] + \\
    + & M_1[L_1\eta, \xi] + L_1[\eta, M_1\xi] - [L_1\eta, M_1\xi] - L_1M_1[\eta, \xi] = 0,
\end{aligned} 
$$
implying that $M_1$ is a symmetry for $L_1$ and thus, completing the proof of the third statement of Theorem \ref{t1}. The fourth statement is proved in a similar way.  Theorem \ref{t1} is proved.


\subsection{Space of symmetries  of a $\gl$-regular Nijenhuis operator:  proof of Theorem \ref{t2_1}}

We start with the following general identity in differential geometry.

\begin{Lemma}\label{lm2}
Assume that operator fields $A, B$ both commute with $C$ (they do not necessarily commute with each other). Then the following formulas hold
$$
\begin{aligned}
    & \langle AB, C \rangle (\xi, \eta) = \langle A, C \rangle (B\xi, \eta) + A \langle B, C \rangle (\xi, \eta), \\
    & \langle C, AB \rangle (\xi, \eta) = \langle C, A \rangle (\xi, B \eta) + A \langle C, B \rangle (\xi, \eta).
\end{aligned}
$$
\end{Lemma}

\begin{proof}
By definition, we have
\begin{equation*}
    \begin{aligned}
    & \langle A, C \rangle (B\xi, \eta) = \underbrace{A[B\xi, C \eta]}_{2} + C[AB\xi, \eta] - \underbrace{AC[B\xi, \eta]}_{1} - [AB\xi, C\eta],\\
    & A \langle B, C \rangle (\xi, \eta) = AB[\xi, C\eta] + \underbrace{AC[B\xi, \eta]}_{1} - \underbrace{A[B\xi, C\eta]}_{2} - ABC[\xi, \eta]
    \end{aligned}
\end{equation*}
Take the sum of these identities. The terms, denoted by the same number, cancel out and  we get
$$
\begin{aligned}
    \langle A, C \rangle (B\xi, \eta) & + A \langle B, C \rangle (\xi, \eta)  = \\
    & = AB[\xi, C \eta] + C[AB\xi, \eta] - ABC[\xi, \eta] - [AB\xi, C\eta] = \langle AB, C\rangle (\xi, \eta),
\end{aligned}
$$
as stated. The second statement follows from the first one and obvious identity $\langle A, B\rangle (\xi, \eta) = - \langle B, A\rangle (\eta, \xi)$.
\end{proof}

Throughout this Section, $L$ denotes a $\gl$-regular Nijenhuis operator.
Recall that each operator $M$ that commutes with $L$ can be written in the form $M = g_1 L^{n - 1} + \dots + g_n \operatorname{Id}$.

\begin{Lemma}\label{lm1}
An operator field $M = g_1 L^{n - 1} + \dots + g_n \operatorname{Id}$ is a symmetry of $L$, if and only if the~functions $g_i$ satisfy the PDE system  
\begin{equation}\label{s1}
\begin{aligned}
& L^* \ddd g_i = \sigma^i \ddd g_1 + \ddd g_{i + 1}, \quad i = 1, \dots, n - 1, \\
& L^* \ddd g_n = \sigma^n \ddd g_1,
\end{aligned}    
\end{equation}
where $\sigma_i$ are the coefficients of the charcteristic polynomial of $L$ as in \eqref{eq:charpol}.

Moreover, $\langle L, M \rangle = 0$, i.e., $M$ is a strong symmetry of $L$.
\end{Lemma}

\begin{proof}
For an arbitrary function $g$ and arbitrary commuting operator fields $A, B$, the following formula holds
$$
\begin{aligned}
    \langle A, g B\rangle (\xi, \eta) & = gB[A\xi, \eta] + A[\xi, g B\eta] - [A\xi, g B\eta] - g AB[\xi, \eta] = \\
    & = g \langle A, B \rangle (\xi, \eta) + \langle \ddd g, \xi \rangle AB \eta - \langle \ddd g, A\xi \rangle B\eta
\end{aligned}
$$
or, equivalently,
\begin{equation}\label{cas}
    \langle A, g B\rangle = g \langle A, B \rangle + \ddd g \otimes AB - A^* \ddd g \otimes B.
\end{equation}
Applying  \eqref{cas} and the fact that $L^n = \sigma^1 L^{n - 1} + \dots + \sigma^n \operatorname{Id}$, we get
$$
\begin{aligned}
    \langle L, M \rangle & = \sum \limits_{i = 1}^n \Big( g_i \langle L, L^{n - i}\rangle + \ddd g_i \otimes L^{n - i + 1} - L^* \ddd g_i \otimes L^{n - i}\Big) = \\
    & = \sum_{i = 1}^n g_i \langle L, L^{n - i} \rangle - \sum_{j = 1}^{n - 1} \Big( L^* \ddd g_j \otimes L^{n - j} - \ddd g_{j + 1} \otimes L^{n - j}\Big) + \ddd g_1 \otimes \Big(\sigma^1 L^{n - 1} + \dots + \sigma^n \operatorname{Id}\Big) - \\
    & - L^* \ddd g_n \otimes \operatorname{Id} = \Big(\sigma^n \ddd g_1 - L^* \ddd g_n\Big) \otimes \operatorname{Id}  + \sum_{j = 1}^{n - 1} \Big(\sigma^i \ddd g_1 + \ddd g_{i + 1} - L^* \ddd g_i\Big) \otimes L^{n - j}.
\end{aligned}
$$
Here we used $\langle L, L^{n - i} \rangle = 0$, which is a direct corollary of Lemma \ref{lm2} for arbitrary (not necessarily $\gl$-regular) Nijenhuis operators.

 Recall that $\xi, L\xi, \dots, L^{n - 1}\xi$ are linearly independent for  almost all  tangent vectors $\xi$. Hence for almost all $\xi$, the relation  $\langle L, M\rangle (\xi, \xi)=0$ implies
$$
\bigl( L^* \ddd g_i - \sigma^i \ddd g_1 - \ddd g_{i + 1}, \,\xi\bigr) = 0 \quad\mbox{and}\quad  \bigl(L^* \ddd g_n - \sigma^n \ddd g_1, \,\xi\bigr)=0,
$$
where $\bigl(\cdot,\cdot\bigr)$ denotes pairing of vectors and covectors.
By continuity,  these relations hold for all $\xi$,  leading to relations \eqref{s1}.   The converse is straightforward.

Furthermore, from the same formulas we see that \eqref{s1} implies that $\langle L, M \rangle = 0$, i.e.,  $M$ is a strong symmetry. \end{proof}

Lemma \ref{lm1} proves the first statement of Theorem \ref{t2_1}. Lemma \ref{lm2} (with $A=M$, $B=R$ and $C=L$) implies that $M R$ is a strong symmetry too. Thus, the second statement of Theorem \ref{t2_1} holds as well.

Let $M$ be a symmetry of $L$, written in the form $M = g_1 L^{n - 1} + \dots + g_n \operatorname{Id}$. For this symmetry define the tensor of type $(1, 2)$
\begin{equation}\label{tm}
T_M = \ddd g_1 \otimes L^{n - 1} + \dots + \ddd g_n \otimes \operatorname{Id}.    
\end{equation}
The next Lemma shows that it satisfies a very important condition.

\begin{Lemma}\label{tensor}
For any symmetry $M$ of $L$ one has
$$
T_M(L\xi, \eta) = T_M (\xi, L\eta)
$$
for all vector fields $\xi, \eta$.
\end{Lemma}
\begin{proof}
Applying identities \eqref{s1} and $L^n = \sigma_1 L^{n - 1} + \dots + \sigma_n \operatorname{Id}$ we get
$$
\begin{aligned}
T_M(L\xi, \eta) & = \ddd g_1(L\xi) L^{n - 1} \eta + \dots + \ddd g_n(L\xi) \eta = \sigma_1 \ddd g_1(\xi) L^{n - 1} \eta + \dots + \sigma_n \ddd g_n(\xi) \eta + \\
& + \ddd g_2 (\xi) L^{n - 1} \eta + \dots + \ddd g_n(\xi) L\eta= \ddd g_1(\xi) L^n\eta + \dots + \ddd g_n(\xi) L\eta = T_M(\xi, L\eta),
\end{aligned}
$$
as required.
\end{proof}

Consider a pair of symmetries $M = g_1 L^{n - 1} + \dots + g_n \operatorname{Id}$ and $R = h_1 L^{n - 1} + \dots + h_n \operatorname{Id}$ of the $\gl$-regular Nijenhuis operator $L$ and the corresponding tensors $T_M, T_R$, defined by \eqref{tm}. Applying Lemma \ref{tensor}, we get that the identities
\begin{equation}\label{id}
    T_M (R\xi, \eta) = T_M(\xi, R\eta), \quad \text{and} \quad T_R(L^k\xi, \eta) = T_R(\xi, L^k\eta)
\end{equation}
hold for all powers $k = 0, 1, 2, \dots$ (the first is true as $R$ is a linear combination of powers of $L$). Using identities \eqref{cas} and \eqref{id}, we perform direct computations
$$
\begin{aligned}
    \langle R, M \rangle (\xi, \eta) & = \langle R, \sum_{i = 1}^n g_i L^{n - i} \rangle (\xi, \eta) = \sum_{i = 1}^n g_i \langle R, L^{n - i} \rangle (\xi, \eta) + T_M(\xi, R \eta) - T_M(R\xi, \eta) = \\
    & = \sum_{i = 1}^n \sum_{j = 1}^n g_i h_j \langle L^{n - j}, L^{n - i}\rangle + \sum_{i = 1}^n g_i \Big( T_R (L^{n - i}\xi, \eta) - T_R (\xi, L^{n - i} \eta)\Big) = 0.
\end{aligned} 
$$
Thus, the third statement of Theorem \ref{t2_1} is proved. Now, consider coordinates $u^1, \dots, u^n$ in which $L$ is in the first companion form \eqref{eq:comp1}. Direct computation shows that the differentials $\ddd u^i$ satisfy \eqref{s1}. Thus, by Lemma \ref{lm1}
$$
M  = u^1 L^{n - 1} + \dots + u^n \operatorname{Id}
$$
is a symmetry, which is obviously regular and centered at $\mathsf p$. The converse is also true. If $M = g_1 L^{n - 1} + \dots + g_n \operatorname{Id}$ is regular symmetry, we can think of $g_i$ as local coordinates. Then relations \eqref{s1}  amount to the fact that in these coordinates, $L$ takes the first companion form. Theorem \ref{t2_1} is proved.


\subsection{Space of conservation laws of a $\gl$-regular Nijenhuis operator:  proof of Theorem \ref{t3_2}}

We will need the following well-known statement, which we prove to keep the work self-contained.

\begin{Lemma}\label{c1}
Let $L$ be a Nijenhuis operator and $\ddd f$ its conservation law. Then the $1$-form $(L^*)^k \ddd f$ is closed for all $k \geq 0$.
\end{Lemma}
\begin{proof}
 The statement is trivial for $k = 0$ and coincides with the definition of a conservation law for $k = 1$. To prove Lemma for $k > 1$, it is sufficient to show that if $L^* \ddd f$ is closed, then $L^{*2} \ddd f$ is closed also. Consider  the Nijenhuis torsion $\mathcal N_L$ as a map from $1$-forms to $2$-forms. The vanishing of $\mathcal N_L$ is equivalent to the following identity for all $1$-forms $\alpha$  (see \cite[Definition 2.5]{nij}):
$$
\ddd(L^{*2} \alpha) (\cdot, \cdot) + \ddd \alpha (L \cdot, L \cdot) - \ddd (L^* \alpha) (L \cdot, \cdot) - \ddd (L^* \alpha) (\cdot, L \cdot) = 0.
$$
Letting $\alpha = \ddd f$ and taking into account that $L^* \alpha$ is closed, we conclude that $L^{*2} \ddd f$ is also closed, as claimed.
\end{proof}

Let us now prove Theorem \ref{t3_2}. Consider a symmetry $M$ and its expansion \eqref{eq:bols2} in powers of $L$. For an arbitrary conservation law of $L$, Lemma \ref{c1} yields $\ddd (L^{*k} \ddd f) = 0$.  Therefore,
$$
\ddd (M^* \ddd f) = \ddd (g_1 L^{*(n - 1)} \ddd f + \dots + \ddd g_n \ddd f) =  \ddd g_1 \wedge L^{*(n - 1)} \ddd f + \dots + \ddd g_n \wedge \ddd f ,
$$
or, equivalently,
$$
\ddd (M^* \ddd f)  (\xi, \eta) =  \bigl( \ddd f  ,  T_M(\zeta, \eta) - T_M(\eta, \zeta)\bigr),
$$
where $T_M$ is defined by  \eqref{tm}. Since $L$ is $\gl$-regular,  the vectors $\xi, L\xi, \dots, L^{n - 1}\xi$ span the tangent space for a generic $\xi$. Taking $\eta = L^i \xi$, $\zeta = L^j \xi$, we see that $T_M(\zeta, \eta) - T_M(\eta, \zeta)$ identically vanishes by Lemma \ref{tensor}. Thus, $\ddd (M^* \ddd f) = 0$.

Choose second companion coordinates $u^1, \dots, u^n$ in which $L$ takes the form \eqref{eq:comp2}. By construction, we have
$$
\begin{aligned}
& L^* \ddd u^i = \ddd u^{i + 1}, \quad i = 1, \dots, n - 1, \\
& L^* \ddd u^n = \sigma_n \ddd u^1 + \dots + \sigma_1 \ddd u^n.
\end{aligned}
$$
The last equation follows from the first $n - 1$ equations. This implies that $u^i$ form a regular hierarchy. The converse is obvious, if we take $f_i$ as coordinates.  Theorem \ref{t3_2} is proved.


\subsection{Description of symmetries and conservation laws at almost all points: proof of Theorem \ref{jordan}}

Our goal is to describe symmetries and conservation laws for the operators $N$ and $U$ given by \eqref{eq:nilpot}.

\weg{Recall the following general fact (essentially a corollary of the relative Poincar\'e Lemma). 

\begin{Lemma}\label{r1}
    In local coordinates $u^1, \dots, u^n$, consider a system of linear PDEs
\begin{equation}\label{k1}
\pd{f}{u^i} = \alpha_i (u^1, \dots, u^n), \quad i = 1, \dots, n - 1.    
\end{equation}
Assume that the compatibility conditions $\pd{\alpha_i}{u^j} - \pd{\alpha_j}{u^i} = 0$ hold,  $1 \leq i < j \leq n - 1$. Then the following integral formula defines a particular solution of  \eqref{k1}:
$$
\bar f(u^1, \dots, u^n) = \int_{0}^{u^1} \alpha_1 (t, u^2, \dots) \ddd t + \int_{0}^{u^2} \alpha_2(0, t, u^3, \dots) \ddd t + \dots + \int_0^{u^{n - 1}} \alpha_n (0, \dots, 0, t, u^n) \ddd t.
$$
The general solution of \eqref{k1} is then given by  
\begin{equation}\label{k2}
f = \bar f + h(u^n),    
\end{equation}
where $h(u^n)$ is an arbitrary function.
\end{Lemma}

\begin{Corollary}\label{crl}
Every solution $f$ of \eqref{k1} is uniquely defined by its restriction on the $n$-th coordinate line, i.e., $v(u^n) = f(0, \dots, 0, u^n)$. \end{Corollary}
}

By Lemma \ref{lm1},  each symmetry $M$ of $N$ is of the form $M = g_1 N^{n - 1} + \dots + g_n \operatorname{Id}$, where $g_i$ satisfy relations \eqref{s1}. As $\sigma_i \equiv 0$, these relations become
\begin{equation}\label{nil}
    \begin{aligned}
    & N^* \ddd g_i = \ddd g_{i + 1}, \quad i = 1, \dots, n - 1, \\
    & N^* \ddd g_n = 0.
\end{aligned}
\end{equation}
The next Lemma deals with solutions of this PDE system.

\begin{Lemma}\label{r2}
Let $g_1, \dots, g_n$ be a solution of \eqref{nil} and $f_i (u^n) = g_i(0, \dots, 0, u^n)$. Then this solution is uniquely defined by $f_i$.
\end{Lemma}
\begin{proof}
As we deal with a linear equation, it is enough to prove that the only solution satisfying the initial condition $f_i(u^n) = 0$,  is the trivial one, that is,
$g_i(u) \equiv 0$, $i = 1, \dots, n$.

We start with the last equation $N^* \ddd g_n = 0$ which reads $\pd{g_n}{u^1} = \dots = \pd{g_n}{u^{n - 1}} = 0$.  This means that $g_n$ is a function of the variable $u^n$ only,  and therefore  $g_n(u^1,\dots, u^n) = g_n(0,\dots, 0, u^n)= f_n(u^n)=0$. Next,  consider the equation $N^* \ddd g_{n-1} = \ddd g_n$.  Since $g_n(u)=0$, we get $N^* \ddd g_{n-1} = 0$ and the same argument shows that $g_{n-1}(u) = 0$.  Repeating this procedure for $g_{n-2}, g_{n-3},\dots$ completes the proof.   \end{proof}

Lemma \ref{r2} shows that each symmetry of $N$ is uniquely defined by its initial condition. We now show that for any initial conditions such a solution exists.  

Let $h$ be an arbitrary smooth function of one variable and consider $h(U)$ as defined in  \eqref{hf}.  By construction of functions $h_i$ for $h(U)$ from formula \eqref{hf}, we have
$$
\begin{aligned}
    & \pd{}{u^i} h\big(u^n + \lambda u^{n - 1} + \dots + \lambda^{n - 1} u^1\big) = \lambda^{n - i} h'\big(u^n + \lambda u^{n - 1} + \dots + \lambda^{n - 1} u^1\big) = \\
    = & \lambda^{n - i} \pd{}{u^n} h\big(u^n + \lambda u^{n - 1} + \dots + \lambda^{n - 1} u^1\big) = \lambda^{n - i} \Big( \pd{h_n}{u^n} + \lambda \pd{h_{n - 1}}{u^n} + \dots + \lambda^{n - 1} \pd{h_1}{u^n} + \dots \Big) = \\
    = & \pd{h_n}{u^i} + \lambda \pd{h_{n - 1}}{u^i} + \dots + \lambda^{n - 1} \pd{h_1}{u^i} + \dots
\end{aligned}
$$
This implies that  $h_i$ depends on $u^{n - i + 1}, \dots, u^i$ and the derivatives of $u^1,\dots, u^n$ satisfy $\pd{h_j}{u^i} = \pd{h_{j - 1}}{u^{i - 1}}$. In particular, this implies that the functions $h_i$ satisfy \eqref{nil} and, therefore, $h(U) = h_1 N^{n - 1} + \dots + h_n \operatorname{Id}$ is a symmetry of $N$. 

Next consider the operator $M$ defined by the formula from Theorem \ref{jordan}: 
$$
M = f_1(U) N^{n - 1} + \dots + f_n(U) = g_1 N^{n - 1} + \dots + g_n \operatorname{Id}.
$$
Since $f_i(U)$ and $N$ are symmetries of $N$, then by Theorem \ref{t2_1}, $M$ also defines a symmetry. By Lemma \ref{lm1}, $g_1, \dots, g_n$ satisfy  \eqref{nil}. For $u^1 = \dots = u^{n - 1} = 0$ we get that $U = u^n \operatorname{Id}$ and $f_i (U) = f_i(u^n) \operatorname{Id}$. Hence, on the $n$-th coordinate line, we have
$$
f_1(u^n) N^{n - 1} + \dots + f_n(u^n) \operatorname{Id} = g_1(0, \dots, 0, u^n) N^{n - 1} + \dots + g_n (0, \dots, 0, u^n) \operatorname{Id}
$$
or, equivalently,  $g_i(0, \dots, 0, u^n) = f_i(u^n)$. Thus, we have constructed a symmetry $M$ for an arbitrary initial condition and, as a result, the formula from Theorem \ref{t2_1} describes all the symmetries of $N$.

Now pick a conservation law and denote it as $\ddd g_1$. By Lemma \ref{c1}, there locally exist $g_i$, $i = 2, \dots, n$, such that $\ddd g_{i} = N^* \ddd g_{i - 1}$. At the same time $N^n = 0$, so that in this chain of functions we have $N^* \ddd g_n = 0$. Thus, $g_i$ satisfy \eqref{nil} and the operator field $M = g_1 N^{n - 1} + \dots + g_n \operatorname{Id}$ defines a symmetry. As a result any conservation law of $N$ can be obtained as a coefficient in front of $N^{n - 1}$. Thus, Theorem \ref{jordan} is proved for $N$.

Notice that $U$ itself is a symmetry of $N$. Hence, by Theorem  \ref{t2_1} (item 3), every symmetry of  $N$ is, at the same time, a symmetry of $U$.  But in this statement we can simply interchange $U$ and $N$, since these operators are both $\gl$-regular  (outside of the set $u_{n-1}=0$).  This implies that $U$ and $N$ share the same symmetries.
Finally,  by Theorem \ref{t3_2} (item 1),  every conservation law of $N$ is a conservation law of $U$ and vice versa. Theorem \ref{jordan} is proved.


\subsection{Relation between symmetries and conservation laws:  proof of Theorem \ref{t6}}

The first statement of Theorem \ref{t6}  is a combination of Theorem \ref{t2_1} (item 4)  with Theorem 1.1 from \cite{nij3} on  reducibility of Nijenhuis operators to the first companion form in the real analytic case, and  Theorem \ref{t3_2} (item 2) with   Theorem 1.2 from \cite{nij3} on  reducibility of Nijenhuis operators to the second companion form in the real analytic case. 

The first part of the second statement is a direct corollary of Theorem \ref{t2_1} in the analytic category.

Taking the coefficients of $U$ in  \eqref{eq:bols2} as coordinates, we get that $L = L_{\mathsf{comp1}}$ due to Theorem \ref{t2_1}. For arbitrary symmetry $M = g_1 L^{n - 1} + \dots + g_n \operatorname{Id}$ the equations \eqref{s1} can be written as 
$$
\pd{g}{u} L = L_{\mathsf{comp1}} \pd{g}{u},
$$
where $\pd{g}{u}$ is the Jacobi matrix of functions $g_1, \dots, g_n$. In given coordinates, this equation takes the form 
$$
\pd{g}{u} L_{\mathsf{comp1}} = L_{\mathsf{comp1}} \pd{g}{u}.
$$
Equivalently, it is written as a linear PDE system
$$
\begin{aligned}
    & \pd{g}{u^{i - 1}} = L \pd{g}{u^i}, \quad i = 2, \dots, n, \\
    & \sigma_1 \pd{g}{u^1} + \dots + \sigma_n \pd{g}{u^n} = L \pd{g}{u^1},
\end{aligned}
$$
where $\pd{g}{u^i}$ is the $i$-th column of the Jacobi matrix. One can see that the last equation, in fact, follows from the previous $n - 1$ ones. Thus, we arrive to the following PDE system
$$
\pd{g}{u^{n - i}} = L^i \pd{g}{u^n}, \quad i = 1, \dots, n - 1,
$$
which is a necessary and sufficient condition for $M$ to be a symmetry of $L$.
The initial condition for such  a system is given by at most $n$ functions of a single variable $v_i(u^n) = g_i (0, \dots, 0, u^n)$, analytic in a neighbourhood of zero. Take the symmetry
$$
\tilde M = v_1(U) L^{n - 1} + \dots + v_n(U) = \tilde g_1 L^{n - 1} + \dots + \tilde g_n \operatorname{Id}
$$
and compare it with $M$. For $u^1 = \dots = u^{n - 1} = 0$, we get $U = u^n \operatorname{Id}$ and $v_i(U) = v_i(u^n) \operatorname{Id}$. Thus, due to the $\gl$-regularity of $L$, we get that $\tilde g_i (0, \dots, 0, u^n) = v_i(u^n)$. Thus, due to Cauchy-Kovalevskaya theorem $M = \tilde M$. Thus, the second statement of the Theorem is completely proved.

Finally, let us proceed with the third statement. First,  consider $f=f_1$ from the statement and the corresponding regular hierarchy of conservation laws $\ddd f_1, \dots, \ddd f_n$. If we take them as coordinates $u^1, \dots, u^n$, then by Theorem \ref{t3_2},  $L$ is in the second companion form. Let $h_1$ be an arbitrary conservation law and define the corresponding hierarchy $\ddd h_1, \dots, \ddd h_n$. Following similar steps, we see that the condition for the collection of $h_i$ to form a hierarchy is written as
$$
\pd{h}{u} L_{\mathsf{comp2}} = L_{\mathsf{comp2}} \pd{h}{u}.
$$
Equivalently, it is written as a linear PDE system
$$
\begin{aligned}
    \pd{h}{u^{n - i}} = A_i \pd{h}{u^n}, \quad i = 1, \dots, n - 1,
\end{aligned}
$$
where $A_i$ are defined by \eqref{rec}.
Similarly to the previous reasoning, we obtain that every regular hierarchy is uniquely defined by its initial conditions $w_i (u^n) = h_i (0, \dots, 0, u^n)$. 

At the same time by definition of the hierarchy, we get $L^{*i} \ddd h_1 = \ddd h_{i + 1}$ for $i = 0, \dots, n - 1$. We get
$$
\pd{h_i}{u^n} = \langle \ddd h_i, \partial_n \rangle = \langle \ddd h, L^{i - 1} \partial_n \rangle = \pd{h}{u^{n - i + 1}} + \sigma_i \pd{h}{u^{n - i + 2}} + \dots
$$
Here $\dots$ stands for the terms containing $\pd{h}{u^n}, \dots, \pd{h}{u^{n - i + 3}}$ with coefficients being polynomial in $\sigma_1, \dots \sigma_{i - 1}$. We get that this is a triangular systems and collection $\pd{h_1}{u^n}, \dots, \pd{h_n}{u^n}$ is uniquely expressed in terms of $\pd{h}{u^1}, \dots, \pd{h}{u^n}$ and vice versa. In particular, the hierarchy is uniquely defined (up to the choice of constants $h_i (\mathsf p) = c_i$) by the restriction of $\ddd h$ onto the $n$-th coordinate line, that is $u^1 = \dots = u^{n - 1} = 0$.

Now, take $M$ in the form \eqref{s2}. By Theorem \ref{t3_2}, the $1$-form $M^* \ddd u^1$ is closed and, thus, locally exact. We get that there exists $\tilde h$ such that
$$
\ddd \tilde h = M^* \ddd u^1 = v_1(U)^* \ddd u^n + \dots + v_n(U)^* \ddd u^1. 
$$
Restricting $U$ onto the $n$-th coordinate line, we get $u^n \operatorname{Id}$ and, thus, by construction $\pd{\tilde h}{u^i} (0, \dots, 0, u^n) = v_{n - i + 1} (u^n)$. Thus, for any initial conditions the corresponding conservation law exists and is unique. The theorem is proved.


\section{Proofs: Applications}

\subsection{Conservation laws and symmetries of the operators 
$\det(L -\lambda \Id) (L- \lambda\Id)^{-1}$:  proof of Theorem \ref{t4}}

\begin{Lemma}\label{m1}
    Let $L: V^n \to V^n$ be a $\gl$-regular operator on a vector space $V^n$ of dimension $n$ and  $A_i$ be defined by  \eqref{rec}. Then the following are equivalent
\begin{enumerate}
    \item A vector $\xi$ is cyclic for $L$.
    \item Vectors $A_0 \xi, A_1 \xi, \dots, A_{n - 1} \xi$ are linearly independent.
\end{enumerate}
\end{Lemma}
\begin{proof} Notice that the collections of operators  $\Id, L, L^2, \dots, L^{n-1}$ 
and $A_0, A_1, \dots, A_{n-1}$  both span the centraliser of $L$ (and, in fact, form bases of the centraliser).   Hence,  
$$
\operatorname{Span}(A_0 \xi, A_1 \xi, \dots, A_{n - 1} \xi)= 
\operatorname{Span}( \xi, L \xi, \dots, L^{n - 1} \xi).
$$
It remains to notice that Statement 1 (resp. 2) is equivalent to $\dim \operatorname{Span}( \xi, L \xi, \dots, L^{n - 1} \xi) =n$ (resp.  
$\dim \operatorname{Span}(A_0 \xi, A_1 \xi, \dots, A_{n - 1} \xi) =n$). \end{proof}
 
Since $A_i$'s form a basis of the centraliser of $L$, any operator field that commutes with all $A_i$ can be written in the form
$$
B = f_1 A_{n - 1} + \dots + f_n A_0.
$$
We need to show that $B$ is a common symmetry of $A_k$'s if and only if $\ddd f_i$ form an hierarchy of conservation laws of $L$.

By direct computation we get
$$
\begin{aligned}
    \langle A_1, B \rangle & = \sum_{i = 1}^n f_i \langle A_{n - i}, A_1 \rangle + \sum_{j = 1}^{n} \Big( A_1^* \ddd f_j \otimes A_{n - j} - \ddd f_j \otimes A_1 A_{n - j}\Big) = \\
    & = \sum_{i = 1}^n f_i \langle A_{n - i}, A_1 \rangle + \sum_{j = 1}^{n} \Big( (A_1^* \ddd f_j + \sigma_1 \ddd f_j) \otimes A_{n - j} - \ddd f_j \otimes L A_{n - j}\Big) = \\
    & = \sum_{i = 1}^{n} f_i \langle A_{n - i}, A_1 \rangle + \sum_{j = 1}^{n - 1} (L^* \ddd f_j - \ddd f_{j + 1}) \otimes A_{n - j} + \Big(L^* \ddd f_n - \sigma_1 \ddd f_n - \dots - \sigma_n \ddd f_1\Big) \otimes A_0.
\end{aligned}
$$
In \cite{m} it was shown that $\langle A_{n - i}, A_1 \rangle (\xi, \xi) = 0$ (see also \cite{nij3}). Thus, we get
$$
\begin{aligned}
   \langle A_1, B \rangle (\xi, \xi) & = \bigl( L^* \ddd f_n - \sigma_1 \ddd f_n - \dots - \sigma_n \ddd f_1, \xi \bigr)\, A_0\xi + \sum_{j = 1}^{n - 1} \bigl( L^* \ddd f_j - \ddd f_{j + 1}, \xi\bigr)\, A_{n - j}\xi,
\end{aligned}
$$
where $\bigl(\cdot, \cdot\bigr)$ denotes the paring of vectors and covectors.
Since cyclic vectors are generic, then by Lemma \ref{m1} the vectors $A_0\xi, \dots, A_{n - 1}\xi$ are linearly independent for almost all $\xi$. Hence, if $B$ is a symmetry of $A_1$, then all the coefficients in the r.h.s. are zero for almost all $\xi$ and, by continuity, for all $\xi$. Thus, we come  to the following system of relations
\begin{equation}\label{dst}
\begin{aligned}
& L^* \ddd f_j = \ddd f_{j + 1}, \quad j = 1, \dots, n - 1, \\
& L^* \ddd f_n = \sigma_1 \ddd f_n + \dots + \sigma_n \ddd f_1,
\end{aligned}    
\end{equation}
which means that $\ddd f_i$ form an hierarchy of conservation laws of $L$.  This shows, in particular, that each common symmetry $B$ of  $A_k$'s has the form $B=f_1 A_{n-1} + \dots + f_n A_0$ with $\ddd f_1, \dots, \ddd f_n$ being an hierarchy of conservation laws.
 
Now let us show that such $B$ is indeed a symmetry for all $A_k$, not only $A_1$. First, consider the $(1, 2)$-tensor $T = \ddd f_1 \otimes A_{n - 1} + \dots + \ddd f_n \otimes A_0$. Taking \eqref{dst} into account, for an arbitrary vector field $\xi$ we get 
$$
\begin{aligned}
T(L\xi, \xi) & = \ddd f_1 (L\xi) A_{n - 1}\xi + \dots + \ddd f_{n} (L\xi) A_0\xi = L^*\ddd f_1(\xi) A_{n - 1} \xi + \dots + L^* \ddd f_n (\xi) A_0\xi\\
& = \ddd f_2 (\xi) A_{n - 1}\xi + \dots + \ddd f_n(\xi) A_1 \xi + \big(\sigma_1 \ddd f_n(\xi) + \dots + \sigma_n \ddd f_1(\xi)\big) A_0 \xi = \\
& = \ddd f_n(\xi) \big(A_1\xi + \sigma_1 A_0 \xi \big) + \dots + \ddd f_2(\xi) \big(A_{n - 1}\xi + \sigma_{n - 1} A_0 \xi \big) + \sigma^n \ddd f_1(\xi) A_0 \xi = \\
& = \ddd f_n(\xi) L\xi + \dots + \ddd f_2 (\xi) L A_{n - 2} \xi + \ddd f_1(\xi) LA_{n - 1} \xi = T(\xi, L\xi).
\end{aligned}
$$
In the last step of calculations, we used  \eqref{rec} and the fact that $L A_{n - 1} - \sigma_n \operatorname{Id} = 0$ due to the Cayley--Hamilton theorem. This identity automatically implies that $T(L^k\xi, \xi) = T(\xi, L^k\xi)$ for all powers $k$ and, as a result, $T(A_k\xi, \xi) = T(\xi, A_k\xi)$ for $k = 2, \dots, n - 1$. Again, by direct computation we get
$$
\langle A_{k}, B \rangle (\xi, \xi) = \sum_{i = 1}^n f_i \langle A_k, A_i\rangle (\xi, \xi) + T(A_k \xi, \xi) - T(\xi, A_k \xi) = 0.
$$
Thus, $B$ is a symmetry for $A_k$ if and only if $\ddd f_i$ form an hierarchy of conservation laws. The first statement of Theorem \ref{t4} is proved. 
 
Now, let us proceed to the second statement of Theorem \ref{t4}. Let $\ddd g_1$ be a common conservation law for $A_k$, $k = 1, \dots, k - 1$. In other words,  there exist $\ddd g_2, \dots, \ddd g_n$ such, that $A_k^* \ddd g_1 = \ddd g_{k + 1}$. The first relation yields
$$
\begin{aligned}
 A_1^* \ddd g_1 = L^* \ddd g_1 - \sigma_1 \ddd g_1 = \ddd g_2. \\   
\end{aligned}
$$
One can see that this is exactly the first relation of \eqref{s1}. Applying \eqref{rec}, we get by induction, that $\ddd g_1, \dots, \ddd g_n$ satisfy the system \eqref{s1}. In particular, due to Lemma \ref{lm1} we get that conservation laws $\ddd g_1$ are  in one-to-one correspondence with symmetries $M = g_1 L^{n - 1} + \dots + g_n \operatorname{Id}$ of Nijenhuis operator $L$. The theorem is proved.


\subsection{Integration of system \eqref{sys} for given initial conditions: proof of Theorem \ref{t5}}

We will verify the statement of Theorem  \ref{t5}  in a special coordinate system, by taking local coordinates $u^1,\dots, u^n$ to be the  functions $g_1, \dots, g_n$ themselves.  As we know, in this coordinate system the operator $L$ is in the second companion form.

\begin{Lemma}\label{im1}
Let $A_i$ be a series of operators defined by  \eqref{rec}. Assume that in coordinates $u^1, \dots, u^n$, the operator $L$ is in the second companion form \eqref{eq:comp2}. Then $A_i \partial_n = \partial_{n - i}$.
\end{Lemma}
\begin{proof}
From \eqref{eq:comp2} we have
\begin{equation}\label{rp}
\begin{aligned}
    & L_{\mathsf{comp2}}\, \partial_i = \partial_{i - 1} + \sigma_{n - i + 1} \partial_n, \quad i = 2, \dots, n, \\
    & L_{\mathsf{comp2}}\, \partial_1 = \sigma_n \partial_n.
\end{aligned}    
\end{equation}
First, notice that the statement of the Lemma holds for $A_0$, that is, we have $A_0 \partial_n = \partial_n$. Now, we proceed by induction, applying formulas \eqref{rec} and \eqref{rp}
$$
A_{i} \partial_n = A_{i-1} L_{\mathsf{comp2}} \,\partial_n - \sigma_{i-1} \partial_n =
 L_{\mathsf{comp2}} A_{i-1} \,\partial_n - \sigma_{i-1} \partial_n  =
 L_{\mathsf{comp2}} \,\partial_{n - i+1} - \sigma_{i-1} \partial_n = \partial_{n - i}.
$$
The Lemma is proved.
\end{proof}

In our special coordinates, the systems of equations from Theorem \ref{t5} becomes very simple:    $u^i = t_{n - i}$, $i = 1, \dots, n - 1$ and $u^n = x$. By construction $u_{t_i} = \partial_{n - i}$, $u_x = \partial_n$. Thus, we get that $u_{t_{i + 1}} = A_i u_x$. In other words, we have constructed a solution of system \eqref{sys} in the coordinates $u^1, \dots, u^n$ with initial condition $(0, \dots, 0, x)$, which  coincides with the initial curve $\gamma(x)$. Thus, the first part of Theorem \ref{t5} is proved.

Now consider an arbitrary regular hierarchy $f_1, \dots, f_n$. Let $M$ be the symmetry of $L$ satisfying the conditions of the Theorem, that is, $M(\gamma(x))\gamma'(x)=\xi(x)$.   By Theorem \ref{t3_2},  the $1$-forms $M^* \ddd f_i$ are all closed. In particular, there exists a unique hierarchy $g_i$ such that $g_i (\mathsf p) = 0$ and $\ddd g_i = M^* \ddd f_i$. For the initial curve $\gamma(x)$ by direct computation we get
$$
\begin{aligned}
    \frac{\ddd}{\ddd x} g_i\bigl(\gamma(x)\bigr) = \bigl( \ddd g_i, \gamma'(x) \bigr) = \bigl( M^* \ddd f_i, \gamma'(x) \bigr) = \bigl( \ddd f_i, M\gamma'(x) \bigr) = \bigl( \ddd f_i, \xi(x) \bigr) = \delta_{i n}.
\end{aligned}
$$
Recalling the condition $g_i (\mathsf p) = 0$, we get that $g_1(\gamma(x)) = x$, $g_i (\gamma(x)) = 0, i = 2, \dots, n$. Thus,  Theorem \ref{t5} is proved.


\section{Conclusion}

One of the motivations behind the `Nijenhuis Geometry' research programme proposed in \cite{openprob,nij} is the observation that Nijenhuis operators pop up in many unrelated subjects of algebra, differential geometry, and mathematical physics; so understanding the objects that are natural from the viewpoint of Nijenhuis geometry, will provide effective tools for tackling problems in other branches of mathematics. This also determines the concept of the present paper. In its `theoretical' part, we discuss conservation laws and symmetries of Nijenhuis operators. In particular, Theorem \ref{t6} is the existence result for symmetries with prescribed initial conditions. A similar fact for conservation laws was known before, see e.g., \cite{cgm, gm}; our Theorems \ref{t1},  \ref{jordan} and \ref{t6} describe them explicitly for $\gl$-regular Nijenhuis operators  near almost every point. 

Note that an essential part of the theoretical results of our paper  is build on  the concept of the  first and
second companion forms for $\gl$-regular operators developed in \cite{nij3}.

In the application part, we consider a specific system  \eqref{sys} 
of partial  differential equations of hydrodynamic type.  This system can be obtained  in the framework of \cite{ml}, which implies that the system is compatible. We prove that the  system is integrable in quadratures, that is, for almost every initial data one can find the solution using integration of closed 1-forms and solving systems of functional equations, see Theorem \ref{t5}.  Example \ref{ex:1} demonstrates this for diagonalisable $\gl$-regular Nijenhuis operators. For such operators, our way is essentially equivalent to the one invented in \cite{fer, fer1} and used e.g. in \cite{FF1997,MB2010}. Our methods can be effectively used in the nondiagonalisable case, see Example \ref{ex:2}. This is the first (not-completely-trivial)  system of hydrodynamic type with nondiagonalisable generators, which is integrable in quadratures; the task of constructing such systems was explicitly formulated as  \cite[item  3 in Conclusion]{fer}.  Generally, only few non-diagonalisable integrable systems  of hydrodynamic type are known, see \cite{arsie}, \cite{ferapnondiag}.

We expect that possible applications of  the theoretical results on Nijenhuis operators presented in this paper will go far beyond integrating of system \eqref{sys}. In particular, one can directly employ them in constructing integrable  bi-Hamiltonian systems, following the ideas developed in  \cite{m}, and in constructing interesting hierarchies of integrable systems both in finite- and infinite-dimensional cases, following the approach of \cite{ml}.  We  invite our fellows, mathematicians and physicists,  to join these studies in Nijenhuis geometry and its applications.


\end{document}